\documentclass[12pt]{amsart}

\pdfoutput=1
\usepackage{latexsym, amsmath, amssymb, amsthm, mathrsfs, bm, mathtools,comment, relsize}

\usepackage[table,xcdraw]{xcolor}
\usepackage{subfigure}

\usepackage[mathscr]{eucal}
\usepackage{array}
\usepackage{tikz}
\usetikzlibrary{arrows,shapes,positioning,calc,patterns,cd}
\usepackage{aliascnt}
\usepackage{graphicx} 
\usepackage[T1]{fontenc} 
\usepackage{mathptmx}
\usepackage{microtype}

\usepackage[colorlinks=true,linkcolor=blue,urlcolor=blue, citecolor=blue]{hyperref}
  
  \usepackage[inner=2.3cm,outer=2.3cm, bottom=3.3cm]{geometry}

\usepackage[initials]{amsrefs}
\allowdisplaybreaks

\usepackage{xcolor, color, soul}

\usepackage{color}
\usepackage[all]{xy}  
\usepackage{enumerate}
\usepackage{mathbbol,mathtools}

\renewcommand{\leq}{\leqslant}
\renewcommand{\geq}{\geqslant}

\newtheorem{theorem}{Theorem}[section]

\newaliascnt{headcor}{headthm}

\aliascntresetthe{headcor}

\newaliascnt{headconj}{headthm}

\aliascntresetthe{headconj}

\newaliascnt{corollary}{theorem}
\newtheorem{corollary}[corollary]{Corollary}
\aliascntresetthe{corollary}

\newaliascnt{claim}{theorem}

\aliascntresetthe{claim}

\newaliascnt{lemma}{theorem}
\newtheorem{lemma}[lemma]{Lemma}
\aliascntresetthe{lemma}

\newaliascnt{conjecture}{theorem}

\aliascntresetthe{conjecture}

\newaliascnt{proposition}{theorem}
\newtheorem{proposition}[proposition]{Proposition}
\aliascntresetthe{proposition}

\newtheorem{theoremx}{Theorem}

\theoremstyle{definition}
\newaliascnt{definition}{theorem}
\newtheorem{definition}[definition]{Definition}
\aliascntresetthe{definition}

\newaliascnt{notation}{theorem}
\newtheorem{notation}[notation]{Notation}
\aliascntresetthe{notation}

\newaliascnt{example}{theorem}
\newtheorem{example}[example]{Example}
\aliascntresetthe{example}

\newaliascnt{examples}{theorem}

\aliascntresetthe{examples}

\newaliascnt{remark}{theorem}
\newtheorem{remark}[remark]{Remark}
\aliascntresetthe{remark}

\newaliascnt{fact}{theorem}

\aliascntresetthe{fact}

\newaliascnt{question}{theorem}

\aliascntresetthe{question}

\newaliascnt{questions}{theorem}

\aliascntresetthe{questions}

\newaliascnt{problem}{theorem}

\aliascntresetthe{problem}

\newaliascnt{construction}{theorem}

\aliascntresetthe{construction}

\newaliascnt{setup}{theorem}
\newtheorem{setup}[setup]{Setup}
\aliascntresetthe{setup}

\newaliascnt{algorithm}{theorem}

\aliascntresetthe{algorithm}

\newaliascnt{observation}{theorem}

\aliascntresetthe{observation}

\newaliascnt{defprop}{theorem}

\aliascntresetthe{defprop}

\def\equationautorefname~#1\null{(#1)\null}
\def\sectionautorefname~#1\null{Section #1\null}
\def\subsectionautorefname~#1\null{\S #1\null}

\newcommand{\R}[1]{R^{1/p^{#1}}}
\newcommand{\m}{\mathfrak{m}}

\newcommand{\NN}{\mathbb{N} }
\newcommand{\ZZ}{\mathbb{Z}}

\newcommand{\kk}{\mathbb{k}}

\newcommand{\IN}{\operatorname{in}}

\newcommand{\cR}{\mathcal{R}}

\newcommand{\reg}{\operatorname{reg}}

\newcommand{\Depth}{\operatorname{depth}}

\newcommand{\gr}{\operatorname{gr}}

\DeclareMathOperator{\chara}{{{char}}}

\newcommand{\rank}{\operatorname{rank}}	
\newcommand{\bh}{\operatorname{bigheight}}	
	
\newcommand{\depth}{\operatorname{depth}}	
\newcommand{\Char}{\operatorname{char}}

\newcommand{\height}{\operatorname{height}}	

\newcommand{\Ht}{\operatorname{ht}}

\newcommand{\deta}[1]{\operatorname{det}\left( {#1} \right)}

\makeatletter
\@namedef{subjclassname@2020}{\textup{2020} Mathematics Subject Classification}
\makeatother

\def \R{\mathcal R}

\DeclareMathOperator{\lcm}{{lcm}}

\newcommand{\ls}{\leqslant}%
\newcommand{\gs}{\geqslant}

\newcommand{\MIN}{\operatorname{Min}}

\newcommand{\C}{\mathcal{C}}

\newcommand{\init}{\mathrm{in}}
\newcommand{\ttt}{{\bf t}}

\makeatletter
\@addtoreset{equation}{section}
\makeatother

\author[A. De Stefani]{Alessandro De Stefani}
\address{Dipartimento di Matematica, Universit{\`a} degli Studi di Genova, Via Dodecaneso 35, 16146 Genova, Italy}
\email{alessandro.destefani@unige.it}

\author[J. Montaño]{Jonathan Monta\~no}
\address{School of Mathematical and Statistical Sciences, Arizona State University, P.O. Box 871804, Tempe, AZ 85287-18041, United States of America}
\email{montano@asu.edu}

\author[L. N{\'u}{\~n}ez-Betancourt]{Luis N{\'u}{\~n}ez-Betancourt}
\address{Centro de Investigaci{\'o}n en Matem{\'a}ticas, Guanajuato, Gto., M{\'e}xico}
\email{luisnub@cimat.mx}

\author[L. Seccia]{Lisa Seccia}
\address{Max Planck Institute for Mathematics in the Sciences, Leipzig, Germany}
\email{seccia@mis.mpg.de}

\author[M. Varbaro]{Matteo Varbaro}
\address{Dipartimento di Matematica, Universit{\`a} degli Studi di Genova, Via Dodecaneso 35, 16146 Genova, Italy}
\email{varbaro@dima.unige.it}

\subjclass[2020]{}

\keywords{}

\begin{document}


\title[Ladder determinantal varieties and their symbolic blowups]
{Ladder determinantal varieties and their symbolic blowups} 

\begin{abstract} 
	 In this article we show that  the symbolic Rees algebra of a mixed (two-sided) ladder determinantal ideal is strongly $F$-regular.
	 Furthermore, we prove that the  symbolic associated graded algebra of a  mixed ladder determinantal ideal is $F$-pure. 
Finally, we show that  ideals  of the poset of minors of a generic matrix give rise to $F$-pure algebras with straightening law.
\end{abstract}

\keywords{Ladder determinantal ideals,  symbolic powers, blowup algebras,  $F$-purity, $F$-regularity, Gr\"obner degenerations, matrix Schubert varieties, Knutson ideals.}
\subjclass[2020]{Primary: 14M12, 13C40, 13A35, 13A30. Secondary:  14M15.}

\maketitle

\section{Introduction}\label{Intro}
Let $X=(x_{i,j})$ be a $k\times \ell$ generic matrix. 
A subset $L\subseteq X$ is a  {\it ladder} if whenever  $x_{i,j},x_{i',j'} \in L$ with $i \leq i'$ and $j \geq j'$,  we have $x_{u,v} \in L$ for every $i \leq  u \leq  i'$ and $j' \leq  v \leq  j$ (see \autoref{ladder}).   The {\it unmixed ladder determinantal ideal} $I_t(L)$ associated to $L$ and $t\in \NN$ is the ideal of $\kk[L]$ generated by the $t$-minors of $L$. 
In this article we focus on the more general class of {\it mixed ladder determinantal ideals} $I_\ttt(L)$   which allows the sizes of the minors $\ttt=(t_1,\ldots,t_v) \in \ZZ_{>0}^v$ to vary   
 (see \autoref{sub_ladder}), and on {\it ladder determinantal varieties} which are   the  varieties they determine.

 Unmixed ladder determinantal rings $\kk[L]/I_t(L)$ 
 were first introduced 
 by Abhyankar  in the study of singularities of Schubert varieties 
 \cite{abhyankar1988enumerative}.   Narasimhan \cite{narasimhan1986irreducibility} showed that these rings are integral domains by explicitly  describing Gr\"obner bases for the ideals $I_t(L)$. Later, Herzog and Trung \cite{herzog1992grobner}, and Conca \cite{conca1995ladder} further exploited Gr\"obner techniques to show that these rings are Cohen-Macaulay and normal, respectively.   Conca and Herzog \cite{concaherzog1997} recovered these results by showing more generally that unmixed ladder determinantal rings  have rational singularities if $\kk$ has characteristic zero, or equivalently, that they are $F$-rational if $\kk$ has characteristic $p\gg 0$.  In the same paper, Conca and Herzog point out that it was unknown at the time whether they are $F$-pure in characteristic $p>0$.

A ladder $L$ is  {\it one-sided} if the three `corners' $x_{1,1}, x_{1,\ell}, x_{k,\ell}$ belong to $L$. Gonciulea and Miller  \cite{gonciulea2000mixed} introduced  the class of mixed one-sided ladder determinantal rings  and showed that they also have rational singularities and therefore they are Cohen-Macaulay and normal. Glassbrener and Smith proved that complete intersection one-sided unmixed ladder determinantal rings are  strongly $F$-regular \cite{GS}, and Conca and Herzog later showed that this holds without  the complete intersection assumption  \cite{concaherzog1997} (see also \cite{mehta1985frobenius, hsiao2013f}). On the other hand, one-sided mixed ladder determinantal varieties are  affine neighborhoods of Schubert varieties \cite{fulton1992flags}, thus the fact that they satisfy the above properties can also be deduced from the corresponding results for Schubert varieties \cite{ramanathan1985schubert} (see also \cite[Theorem 2.4.3]{knutson2005grobner}).  In fact, Schubert varieties are globally $F$-regular \cite{lauritzen2006global}.
Two-sided mixed ladder determinantal rings constitute the most general setup; they were defined by Gorla \cite{GorlaMixed}, who showed that they are integral domains and Cohen-Macaulay. In this article we focus on symbolic blowup algebras associated to  
two-sided mixed ladder determinantal ideals. 
As a consequence of our results on symbolic blowups, we obtain that two-sided mixed ladder determinantal ideals define $F$-pure rings (see \autoref{blowups_f_pure}), hence 
answering the question of Conca and Herzog \cite[p. 122]{concaherzog1997}.

Rees algebras associated to a filtration are very important objects in commutative algebra and algebraic geometry. Symbolic Rees algebras, i.e., Rees algebras associated to the filtration of symbolic powers of an ideal, are particularly interesting; for instance, see \cite{GrifoSeceleanu} for a survey on this topic. The fact that they might not be Noetherian rings is related to counterexamples to Hilbert's 14$^{{\rm th}}$ problem, and makes them typically harder to handle than ordinary Rees algebras (i.e., Rees algebras associated to ordinary powers). We also point out that the Noetherianity of symbolic Rees algebras of pure height one ideals has close connections with the minimal model program \cite{MMP1,MMP2,MMPp}, and with the conjectural equivalence between weak and strong $F$-regularity; see \cite{AberbachHunekePolstra} for recent progress in this direction. The Noetherianity of symbolic Rees algebras associated to ladder determinantal ideals was known in the  unmixed case   by  work of Bruns and Conca  \cite[Theorem 4.1]{brunsconca98}, who proved  that they admit a finite SAGBI basis with respect to an  antidiagonal term order. Using a combinatorial operation on ladders, which we call  {\it chamfering}, we show that they are Noetherian also in the  mixed case  (see \autoref{ThmNoetherian}). We point out that our approach does not guarantee that the symbolic Rees algebra associated to a ladder determinantal ideal $I=I_\ttt(L)$ admits a finite SAGBI basis with respect to some monomial order $\prec$. 
However, when it does, as a byproduct of our methods we also conclude that $\IN_{\prec}\left(I^{(n)}\right)=\IN_{\prec}(I)^{(n)}$ for all $n\in\NN$ and $p = \Char(\kk) \gg 0$ (see \autoref{cor_SAGBI}).

Our first main result concerns $F$-singularities of symbolic Rees algebras of ladder determinantal ideals: we show that they are all strongly $F$-regular.

\begin{theoremx}[\autoref{stronglyF}] Let $\kk$ be a perfect field of characteristic $p>0$ and $L$ be a ladder. For $\ttt=(t_1,\ldots,t_v) \in \ZZ_{>0}^v$ let $I_\ttt(L) \subseteq \kk[L]$ be the corresponding mixed ladder determinantal ideal. The symbolic Rees algebra $\R^s\left(I_{\ttt}(L)\right) = \bigoplus_{n \geq 0} I_\ttt(L)^{(n)}$ is strongly $F$-regular.
\end{theoremx}

Next, we turn our attention to symbolic associated graded rings $\gr^s(I_\ttt(L)) = \bigoplus_{n \geq 0} I_\ttt(L)^{(n)}/I_\ttt(L)^{(n+1)}$. Such rings are typically worse behaved than symbolic Rees algebras; however, their singularities have closer connections to those of the quotient $\kk[L]/I_\ttt(L)$, since the latter is a direct summand of $\gr^s(I_\ttt(L))$. The aforementioned result about strong $F$-regularity of two-sided mixed ladder determinantal varieties points to the direction that $\gr^s(I_\ttt(L))$ might always be strongly $F$-regular. While we could not establish whether this is true or not, we prove that $\gr^s(I_\ttt(L))$ is always at least $F$-pure.

\begin{theoremx}[\autoref{ThmMixLadderSFP}]
Let $\kk$ be a perfect field of characteristic $p>0$ and $L$ be a ladder. For $\ttt=(t_1,\ldots,t_v) \in \ZZ_{>0}^v$ let $I_\ttt(L) \subseteq \kk[L]$ be the corresponding mixed ladder determinantal ideal. The symbolic associated graded ring $\gr^s\left(I_{\ttt}(L)\right)$ is $F$-pure.
\end{theoremx}

At the end of Section \ref{MLDSection} we show that two-sided unmixed ladder determinantal ideals are {\it Knutson}  associated to the product $f$ of the minors corresponding to  the main antidiagonals of the ladder. The notion of Knutson ideals  was  introduced   over finite fields  \cite{knutson2009frobenius} and later developed over any field (also of characteristic 0) by Seccia \cite{seccia2021knutson}. More generally we prove that the ideal generated by fixed-size minors in the subladders corresponding to consecutive rows or columns in the matrix are Knutson ideals associated to the  polynomial $f$ described above. In particular, the natural minimal generators of any sum of these ideals form a Gr\"obner basis with respect to an antidiagonal term order. If, besides the operation of taking sums, we also allow the one of taking intersections, all the resulting ideals still admit a squarefree initial ideal with respect to an antidiagonal term order and, in positive characteristic, define $F$-pure rings (\autoref{ThmLadderKI}).

When $L=X$ the above result had already been proved by Seccia  \cite{seccia2022knutson}. In the case $L=X$ we also prove that any ideal of fixed-size minors of a top-left or bottom-right corner submatrix of $X$ is Knutson associated to $f$ (\autoref{p:msk}). As a consequence any sum of these ideals is  Knutson associated to $f$ and  their natural minimal generators  form a Gr\"obner basis with respect  to antidiagonal term orders. If, besides the operation of taking sums, we also allow the one of taking intersections, all the resulting ideals still admit a squarefree initial ideal with respect to an antidiagonal term order and, in positive characteristic, define $F$-pure rings. Some classes of ideals that can be obtained this way and for which we obtain the latter consequences  are Schubert determinantal ideals and, more generally, alternating sign matrix ideals as they are sums of ideals of minors of top-left corner submatrices of $X$ (see	\autoref{c:msk}). We remark that the fact that Schubert determinantal varieties 
are $F$-pure in positive characteristic has been known since the work of Brion and Kumar  \cite{Brion_Kumar_Book}, and they are in fact strongly $F$-regular as already pointed out before. The fact that the natural minimal generators of Schubert determinantal ideals form a Gr\"obner basis with respect  to antidiagonal term orders was also known  by work of Knutson and Miller  \cite{knutson2005grobner}. Moreover,  a proof by induction on the Bruhat order, and so different from the one given here, of the fact that  Schubert determinantal ideals are Knutson associated to $f$ was  already sketched  by Knutson \cite{knutson2009frobenius}. Klein and  Weigandt \cite{klein2021bumpless} showed that alternating sign matrix ideals have radical initial ideals with respect to antidiagonal term orders and hence they are radical themselves. 
All of the above facts are recovered, and sometimes strengthened, by \autoref{c:msk}. 

We conclude by highlighting a connection of our work with  Eisenbud's suspicion that graded Algebras with Straightening Law (ASL for short) might always be $F$-pure in positive characteristic \cite[p. 245]{eisenbud80}. As noticed by Koley and Varbaro \cite[Remark 5.2]{KoleyVarbaro} this is not always true, although  all graded ASL are $F$-injective. We confirm Eisenbud's intuition for ASL arising from ideals of the poset of minors of a generic matrix: since they are intersection of sums of ideals of minors of top-left corner submatrices of $X$, our results yield that they are Knutson. 
\begin{theoremx}[\autoref{c:pif}] \label{THMC}
Let $\kk$ be a field, $X$ be a generic matrix and $\Pi$ be the poset of all minors of $X$. There exists a polynomial $f \in \kk[X]$ such that, for any  ideal $\Omega\subseteq \Pi$, the poset ideal $\Omega\kk[X]$  is Knutson with respect to $f$. In particular, $\init_\prec(\Omega \kk[X])$ is a squarefree monomial ideal for any antidiagonal term order $\prec$, and $\kk[X]/\Omega \kk[X]$ is $F$-pure  in positive characteristic.
\end{theoremx}

An analogous result actually holds for what we call generalized poset ideals, a class which strictly includes poset ideals; see \autoref{c:gpif}.


\section{Background}\label{SectionBackground}

In this section we include some preliminary material that is needed in the subsequent sections. In particular, in   \autoref{sub_ladder} and \autoref{sub_affine_schub} we recall the definitions of the classes of determinantal ideals that are treated in this paper. 

\subsection{Methods in prime characteristic}\label{SubSecPrimeChar} 

We begin by recalling some concepts from positive characteristic commutative algebra. For more information we refer the reader to Ma and Polstra's book \cite{ma2021f}.

\begin{definition}\label{DefFbasic}
Let $R$ be a ring of prime characteristic $p > 0$. For every $e\in \ZZ_{> 0}$, we denote by $F^e_*( R)=\{F^e_* (r)\;|\; r\in R\}$ the $R$-algebra which is isomorphic to $R$ as a ring and whose action by $R$ is given by $f F^e_* (r)=F^e_* \left(f^{p^e}r\right)=F^e_*\left(f^{p^e}\right) F^e_* (r)$. 
We say that $R$ is {\it $F$-finite} if the the Frobenius map $F^e:R\to F_*^e(R)$, $r\mapsto F_*^e(r^{p^e})$ is finite for some, or equivalently all, $e\in \ZZ_{> 0}$, that is, if  $F_*^e(R)$ is a finitely generated  $R$-module. 
%

We say that $R$ is  {\it $F$-pure} if for some, or equivalently all, $e\in \ZZ_{>0}$ the Frobenius map $F^e$ is pure, i.e.,  the map $M\otimes_R R\xrightarrow{\mathbf{1}_M\otimes_R F^e} M\otimes_R F_*^e(R)$ is injective for every $R$-module $M$. 
We say that $R$ is {\it $F$-split} if for some, or equivalently all, $e\in \ZZ_{>0}$ the map $ F^e$ splits. 
Clearly, if $R$ is $F$-split, it is $F$-pure; the converse holds if $R$ is $F$-finite. We also note that if $R$ is  $F$-pure, then it is reduced, and in this case one can identify the map $R \to F^e_*(R)$ with the natural inclusion $R \hookrightarrow R^{1/p^e}$, where $R^{1/p^e}$ is the ring of $p^e$-th roots of $R$. In particular, $R$ is $F$-pure (resp. $F$-split) if and only if the map $R\to R^{1/p^e}$ is pure (resp. splits) for some, or equivalently all, $e\in \ZZ_{>0}$. 
\end{definition}

We now  recall the following notion that connects symbolic powers and the property of being $F$-split \cite{DSMNB}. 

\begin{definition}\label{DefFpureFilt}
 Let $R=\kk [x_1,\ldots,x_d]$ be a standard graded polynomial ring over an $F$-finite  field $\kk$ of prime characteristic. We say that a homogeneous ideal $I\subseteq R$ is  {\it symbolic $F$-split} if  there exists a splitting $\phi:R^{1/p}\to R$ such that
$\phi\left(\left(I^{(np+1)}\right)^{1/p}\right)\subseteq I^{(n+1)}$ for every $n\in \NN$.   
Equivalently, if
$$\bigcap^{\infty}_{n=0} \left[ \left( I^{(n+1)}\right)^{[p]}:I^{(np+1)}\right]\not\subseteq \m^{[p]}.$$
\end{definition}

\begin{remark}\label{rem_impt_prop_symb}
	In the present paper, ideals which are symbolic $F$-split are especially relevant because of the following two properties:
\begin{enumerate}
\item If  $I\subseteq R$ is  symbolic $F$-split, 
then the {\it symbolic Rees algebra} $\R^s(I):=\oplus_{n\in \NN} I^{(n)}$ and the {\it symbolic associated graded algebra}   $\gr^s(I):=\oplus_{n\in \NN} I^{(n)}/I^{(n+1)}$ are $F$-split. 
In particular, if $I$ is symbolic $F$-split then  $R/I$ is also $F$-split \cite[Theorem 4.7]{DSMNB}.
\item Let  $I\subseteq R$ be a homogeneous ideal and $H:=\bh(I)$ the maximum height of a minimal prime ideal of $I$. If there exists a term order $\prec$ on $R$ for which $\init_\prec\left(I^{(H)}\right)$ contains a squarefree monomial, then $I$ is symbolic $F$-split  \cite[Lemma 6.2(1)]{DSMNB}.  Moreover, in this case $\init_\prec(I)$ is automatically radical \cite[Theorem 3.13]{KoleyVarbaro},  and if $I$ is equidimensional then both $\bigoplus_{n\in \NN}\init_\prec(I^{(n)})$ and $\bigoplus_{n\in \NN}\frac{\init_\prec(I^{(n)})}{\init_\prec\left(I^{(n+1)}\right)}$ are $F$-split \cite[Proposition 7.5]{DSMNB}.
\end{enumerate}
\end{remark}

Let $I\subseteq R$ be a homogeneous ideal such that  $\init_\prec(I)$ is radical. In this case, we have that 
$\init_\prec\left (I^{(n)}\right)\subseteq \init_\prec(I)^{(n)}$  for every  $n\in\NN$   \cite[Proposition 5.1]{Sull}, and hence \[\init_\prec(I)^n\subseteq \init_\prec\left (I^{(n)}\right)\subseteq \init_\prec(I)^{(n)} \,\,\,\text{ for every }n\in\NN.\]
Since $\init_\prec(I)$ is radical,  $\init_\prec\left (I^{(n)}\right)=\init_\prec(I)^{(n)}$ if and only if $\init_\prec\left (I^{(n)}\right)$ has no embedded components. This fact allows us to show the following interesting relationship between symbolic powers and initial ideals.
%

\begin{proposition}\label{symbolic-initial}
	 Let $R=\kk [x_1,\ldots,x_d]$ be a standard graded polynomial ring over a perfect field $\kk$ of prime characteristic. 
Let $P\subseteq R$ be a homogeneous prime ideal of height $h$. Fix a term order $\prec$ on $R$ and  assume that $\bigoplus_{n\in \NN}\init_\prec\left(P^{(n)}\right)$ is Noetherian. If $\chara(\kk)\gg 0$,
then the following are equivalent
	\begin{enumerate}[\rm (1)]
		\item $\init_\prec\left(P^{(n)}\right)=\init_\prec(P)^{(n)}$ for every $n\in\NN$ and $\init_\prec(P)$ is radical.
		\item $\init_\prec\left(P^{(h)}\right)=\init_\prec(P)^{(h)}$ and $\init_\prec(P)$ is radical.
		\item $\init_\prec\left(P^{(h)}\right)$ contains a squarefree monomial.
	\end{enumerate}
\end{proposition}
\begin{proof}
	The implication $(1)\Rightarrow (2)$ is trivial. For $(2)\Rightarrow (3)$ we note that  $x_1\cdots x_d$ belongs to $Q^h$  for every minimal prime ideal $Q$ of $\init_\prec(P)$, so $x_1\cdots x_d\in \init_\prec(P)^{(h)}= \init_\prec\left(P^{(h)}\right)$. 
	It remains to show  $(3)\Rightarrow (1)$.  From \autoref{rem_impt_prop_symb}(2) it follows that  $\IN_\prec(P)$ is radical and $\bigoplus_{n\in\NN}\frac{\init_\prec(P^{(n)})}{\init_\prec(P^{(n+1)})}$ is reduced. Thus, $\init_\prec\left(P^{(n)}\right)=\init_\prec(P)^{(n)}$ for every $n\in\NN$  \cite[Corollary 7.10]{DSMNB}. 
\end{proof}

\begin{remark} \label{remark_char}
The lower bound on the characteristic of the field $\kk $ in \autoref{symbolic-initial} can be made precise: if $\bigoplus_{n\in \NN}\init_\prec(P^{(n)})$ is generated as an $R$-algebra by forms of  degrees $b_1,\ldots ,b_m$, then \autoref{symbolic-initial} applies if  $\chara(\kk)>\lcm(b_1,\ldots ,b_m)$.
\end{remark}

\subsection{Knutson ideals}\label{SubSecKnutson}

The notion of {\it Knutson ideals} was coined by  Conca and Varbaro \cite{conca2020square}  after Knutson's work on compatibly split ideals and degenerations \cite{knutson2009frobenius} (see also \cite{brion2007frobenius}).

\begin{definition}\label{K.I.} Let $f \in R= \kk [x_1,\ldots,x_d]$ be a homogeneous polynomial where $\kk$ is a field and assume that  for some term order $\prec$ the leading term of $f$ is a squarefree monomial. 
	We define $\mathcal{C}_f$ to be the smallest set of ideals satisfying the following conditions:
	\begin{enumerate}[(1)]
		\item $(f) \in \mathcal{C}_f$.
		\item  If $I \in \mathcal{C}_f$ then $(I:J) \in \mathcal{C}_f$ for every ideal $J \subseteq R$.
		\item If $I$ and $J$ are in $\mathcal{C}_f$ then $I+J$ and $I \cap J$ must  also be in $\mathcal{C}_f$.
	\end{enumerate} 
	If $I$ is an ideal in $\mathcal{C}_f$, we say that $I$ is a \emph{Knutson ideal associated to} $f$, or simply that it is {\it Knutson}. 
\end{definition} 
If $\kk$ has positive characteristic and $\deg(f)=d$, then $f$ defines a splitting map on $R$  (see \cite[Lemma 4]{knutson2009frobenius} for finite fields and  \cite{seccia2021knutson} in general). 
Knutson ideals are 
compatibly split with respect to this map. Therefore,  Knutson ideals  define $F$-split 
rings in positive characteristic. In addition, Knutson ideals
have squarefree initial ideals  in arbitrary characteristic (see \cite[Theorem 2]{knutson2009frobenius} for finite fields and \cite{seccia2021knutson} in general) and consequently their extremal Betti numbers 
coincide
with  those of their initial ideals \cite{conca2020square}. Furthermore, Gr\"obner bases of Knutson ideals are well-behaved 
with respect to sums as the 
union of Gr\"obner bases of Knutson ideals is a Gr\"obner basis of their sum. This quality makes computations of Gr\"obner bases for some Knutson ideals somewhat  easier than for general ideals. 
All of  these considerations make  Knutson ideals relevant  objects 
in computational algebra, combinatorial commutative algebra, and $F$-singularity theory.  

\begin{remark}
	Since every ideal of $\mathcal{C}_f$ is radical,  Condition  (2)  in  \autoref{K.I.} can be replaced by:
	\begin{itemize}
		\item[$(2^\prime)$] If $I \in \mathcal{C}_f$ then $P \in \mathcal{C}_f$ for every $P \in \MIN(I)$.
	\end{itemize}
\end{remark}

\subsection{Ladder determinantal ideals}\label{sub_ladder}

In this subsection we recall  the definition of two-sided mixed ladder determinantal ideals. For this, we follow closely the notations and conventions form Gorla's work \cite{GorlaMixed} where these ideals were introduced    (see also \cite{gonciulea2000mixed}).

\begin{definition}\label{def_ladder}
	Let $X=(x_{i,j})$ be a generic matrix of size $k \times \ell$,  i.e.,
	\[ X=
	\begin{bmatrix}
		x_{1,1}       & \dots  & x_{1,\ell} \\
		\vdots  & \vdots &\vdots \\
		x_{k,1}        & \dots & x_{k,\ell}
	\end{bmatrix}.
	\]
	We say that a subset $L\subseteq X$ is a {\it {\rm (}two-sided\,{\rm )} ladder} if it has
	the following property: if $x_{i,j},x_{i',j'} \in L$ with $i \leq i'$ and $j \geq j'$,  then $x_{u,v} \in L$ for any $i \leq  u \leq  i'$ and $j' \leq  v \leq  j$.  
	%
	We denote by 
	$$\left\{x_{1,a_1}=x_{b_1,a_1},\ldots,x_{b_u,a_u}=x_{b_u,\ell}\right\}\subseteq L \text{ with } 1=b_1<\cdots < b_u\text{  and  } a_1<\cdots < a_u=\ell$$ 
	the {\it upper-outside corners}, defined by the property that $x_{b_k-1,a_k}\notin L$ and $x_{b_k,a_k+1}\notin L$, and by 
	$$\left\{x_{d_1,1}=x_{d_1,c_1},\ldots,x_{d_v,c_v}=x_{k,c_v}\right\}\subseteq L \text{ with } d_1\ls \cdots \ls d_v=k
	\text{  and } 1=c_1\ls \cdots \ls c_v$$ the {\it lower-outside corners} of $L$.  These are not uniquely determined by $L$, indeed they are part of the definition, even if a subset of the lower-outside corners is determined by $L$: such a subset consists of those $x_{i,j}\in L$ such that $x_{i+1,j}\notin L$ and $x_{i,j-1}\notin L$; consequently, for all $h=1,\ldots v,$ we have that either $x_{d_h+1,c_h}\notin L$ or $x_{d_h,c_h-1}\notin L$. 
	Moreover, we assume  all of the corners are distinct. 
	If $L$ has only one upper-outside corner we say it is {\it one-sided}. 
	We also define the following subladders for all $j=1,\ldots ,v$
\begin{equation}\label{eq_Lj}
	L_j = \{x_{b,a} \in L \mid 1 \leq b \leq d_j, c_j \leq a \leq l\}.
\end{equation}
	In \autoref{example_ladder_1} we illustrate all of these notations in a specific example. 
\end{definition}

\begin{example}\label{example_ladder_1}
	Let $X$ be a generic matrix of size $10\times 10$. In \autoref{ladder} the shaded area represents the ladder $L$ whose upper-outside corners are $\{x_{1,2},\,x_{4,8},\, x_{8,10}\}$ and lower-outside corners are $\{x_{3,1},\,x_{6,1},\, x_{8,6},\, x_{10,8}\}$; note that the first two lower-outside corners belong to the same column. 
	\begin{figure}[ht]
		\includegraphics[scale=0.4]{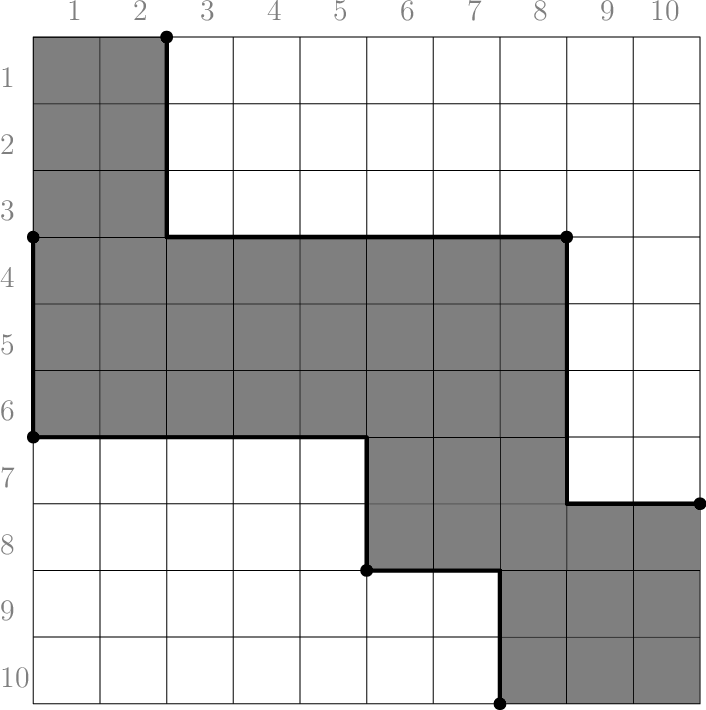}
		\caption{Ladder  $L$}
		\label{ladder}
	\end{figure}
	
	\noindent Since $L$ has four lower-outside corners, it has four subladders $L_1,L_2, L_3, L_4$ as defined above. 
	In \autoref{subladders} we highlight these subladders. 
	\begin{figure}[ht]
		\centering
		\mbox{\subfigure{\includegraphics[width=1.5in]{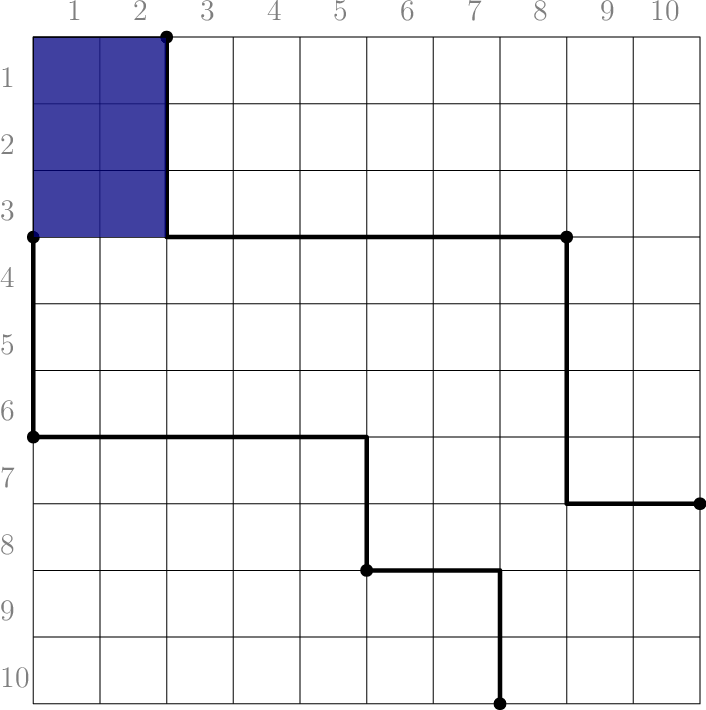}}\quad
			\subfigure{\includegraphics[width=1.5in]{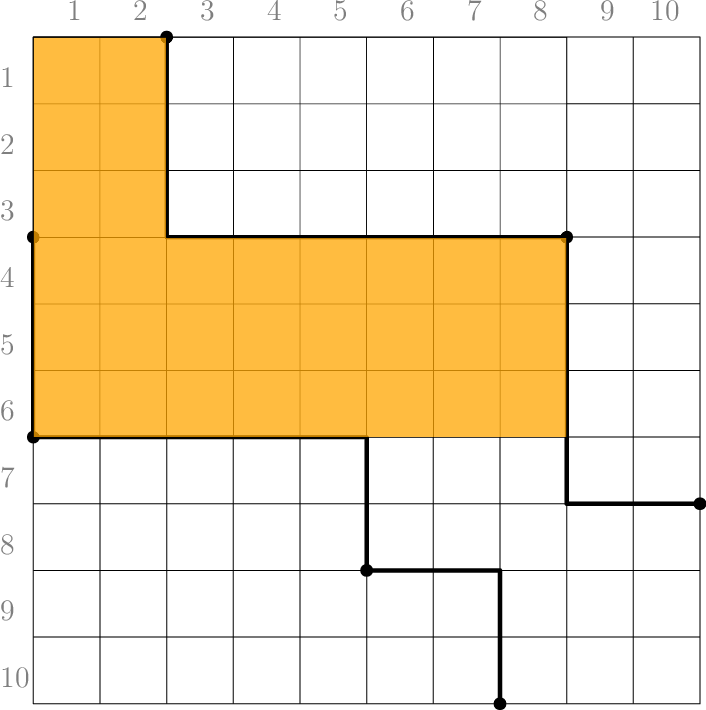} } \quad
			\subfigure{\includegraphics[width=1.5in]{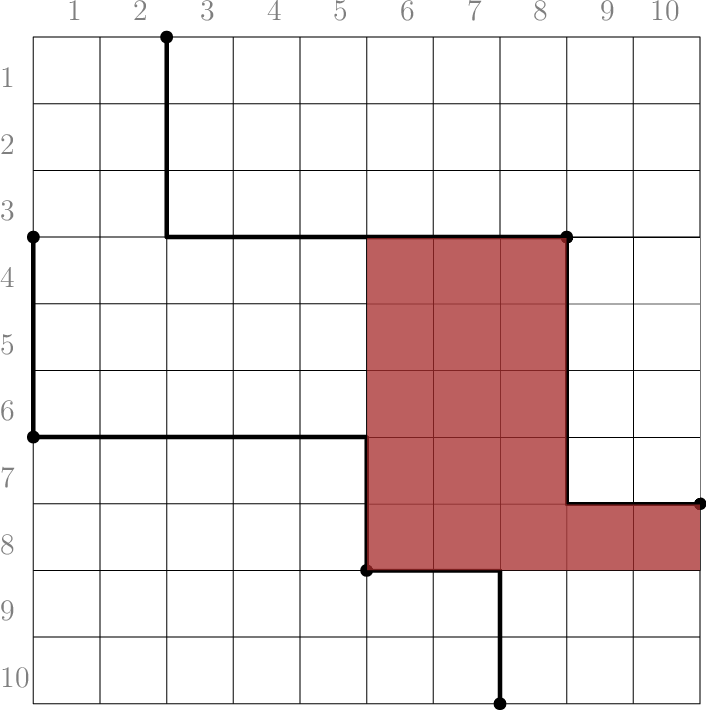} }\quad
			\subfigure{\includegraphics[width=1.5in]{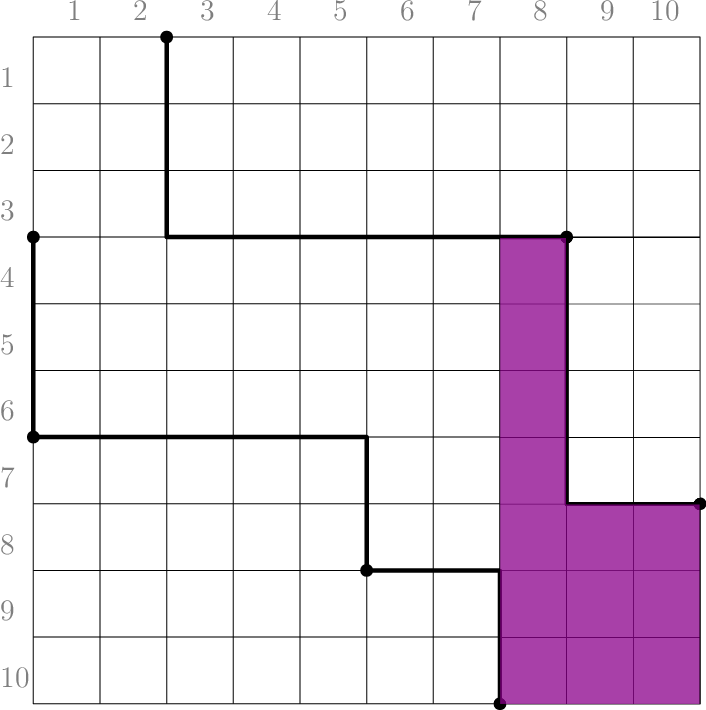} }
		}
		\caption{The subladders $L_1$, $L_2$, $L_3$, $L_4$ ordered from left to right} \label{subladders}
	\end{figure}
\end{example}

\begin{definition}\label{def_mixed_lad_det}  Let $\kk$ be a field and $L \subseteq X$  a ladder as in \autoref{def_ladder}. If  $t$ is a positive integer, we denote by $I_t (L)$ the ideal of $\kk[L]$ generated by the  $t$-minors of $X$ that are completely contained in $L$. The ideal $I_t (L)$ is a {\it ladder determinantal ideal} of $L$. On the other hand, if $\ttt=(t_1,\ldots,t_v) \in \ZZ_{>0}^v$ we denote by $I_\ttt(L)$ the ideal $I_{\ttt}(L) = I_{t_1}(L_1) + \ldots + I_{t_v}(L_v)$. Such ideal is a {\it mixed ladder determinantal ideal} of $L$. Notice that,  if $t=t_1=\cdots=t_v$, then  $I_{\ttt}(L)=I_t (L)$. The ideals $I_{\ttt}(L)$ and $I_{t}(L)$ are prime (see \cite[]{GorlaMixed}).
	
	Throughout this article, for every mixed ladder determinantal ideal  we assume that  all the entries of $L$ are involved in some generator of  $I_{\ttt}(L) $ and that  $I_{t_i}(L_i) \not\subseteq I_{t_j}(L_j)$ for all $i \ne j$.
\end{definition}

\begin{example}
	Let $L$ be as in \autoref{example_ladder_1}. If $\ttt=(2,3,1,2)$, then $I_\ttt(L)=I_2(L_1)+
	I_3(L_2)+
	I_1(L_3)+
	I_2(L_4).$  If $\ttt=(2,2,2,2)$, then $I_\ttt(L)=I_2(L)$.
\end{example}

\begin{notation} \label{running assumptions ladder} We summarize here the assumptions on the ladders treated in the present paper, which are the same as in  Gorla's work \cite{GorlaMixed}.
	\begin{enumerate}[(1)]
		\item All the entries of $L$ are involved in some minor, otherwise we delete the superfluous entries and consider the remaining  ladder. 
		\item No two consecutive corners coincide, and no two upper-outside corners belong to the same row or to the same column.
		\item  $I_{t_i}(L_i) \not\subseteq I_{t_j}(L_j)$ for all $i \ne j$.
		\item When we view $L$ as a ladder inside the $k \times \ell$ matrix $X$ we assume that there is no proper submatrix $X'$ of $X$ which contains $L$.
	\end{enumerate}
\end{notation}


\subsection{Schubert determinantal ideals and  poset ideals}\label{sub_affine_schub}

In this  subsection we introduce further families of determinantal ideals treated in this paper. 
We begin by introducing some notation. 

\begin{notation}\label{notation_affine_subsec}
	Let $X$ be as in \autoref{def_ladder}.	For any $1 \leq a <b \leq \ell$ and $1 \leq c<d \leq k$, we denote by $X_{[a,b]}^{[c,d]}$ the submatrix of $X$ with column indices $a,\ldots, b$ and row indices $c,\ldots, d$. 
	In the case $[c,d]=[1,k]$ (resp. $[a,b]=[1,\ell]$ ) we simply write $X_{[a,b]}$ (resp. $X^{[c,d]}$).  
	We use $X_{a \times b}$ to denote the northwest submatrix $X_{[1,b]}^{[1,a]}$ and $X^{a \times b}$ to denote the southeast submatrix $X_{[\ell-b+1 ,\ell]}^{[k-a+1,k]}$. 
	 Moreover, for a set of indices $1\ls i_1<\cdots<i_r\ls k$ and $1\ls j_1<\cdots<j_r\ls \ell$ we denote by $[i_1\cdots i_r\mid j_1\cdots j_r]$ the determinant ($r$-minor) of the submatrix of $X$ corresponding to rows $i_1,\ldots, i_r$ and columns  $j_1,\ldots, j_r$.
\end{notation}

Following \cite{MillerSturmfels}, by the {\it antidiagonal term} of a minor $[i_1\cdots i_r\mid j_1\cdots j_r]$ we mean the product of the variables in the antidiagonal of the corresponding submatrix, namely $x_{i_1,j_r}x_{i_2,j_{r-1}}\cdots x_{i_r,j_1}$. Let $\kk$ be an arbitrary  field.  A term order on the polynomial ring $\kk[X]$ is {\it antidiagonal} if the initial term of any minor is  its antidiagonal term. A typical example of an antidiagonal term order is the lexicographical one with variables ordered as $$x_{1,\ell}>\cdots>x_{1,1}>x_{2,\ell}>\cdots>x_{2,1}>\cdots >x_{k,\ell}>\cdots>x_{k,1}.$$

Although it is not our intention to review the theory of {\it algebras with straightening law} in this paper (see \cite[Chapter 3]{bruns2022determinants}),  we wish to introduce a few concepts for the sake of clarity. The set $\Pi$ of all the minors of the generic matrix $X$ can be partially ordered by the relation
\[[a_1\ldots a_r|b_1\ldots b_r]\leq [c_1\ldots c_s|d_1\ldots d_s] \iff r\geq s, \ \ a_i\leq c_i \mbox{ and } b_i\leq d_i \ \text{ for all } \ i=1,\ldots ,s.\] 
 The starting observation is that $\kk[X]$  is generated 
as a 
$\kk$-vector space by the set 
by
\[\{\pi_1\cdots \pi_m\mid  \ \pi_i\in\Pi \text{ for }i=1,\ldots, m \text{ and } \ \pi_1\leq \cdots \leq \pi_m\}.\]
The elements of this $\kk$-basis are called {\it standard monomials}. Although it may happen that the product of two standard monomials is not a standard monomial, such a product
can be uniquely expressed as a $\kk$-linear combination of standard monomials;   
 this process  is known in the literature as {\it straightening law}. 
 
With the above notation, we are ready to define matrix Schubert varieties and its defining ideals (see \cite[Chapter 15]{MillerSturmfels} for more information).

\begin{definition}\label{def_MSV}
Let $\mathscr{M}_{k\times \ell}(\mathbb{\kk})$ be the vector space of $k\times \ell$ matrices with entries in $\kk$ and let $w=(w_{i,j})\in \mathscr{M}_{k\times \ell}(\mathbb{\kk})$ be a {\it partial permutation}, i.e., $w_{i,j}=0$ for all $i,j$ except for at most one entry equal to 1 in each row and  column.
\begin{enumerate}[(1)]
\item The {\it matrix Schubert variety} $\overline{X_w}$ associated to $w$ is given by 
$$\overline{X_w}:=\left\{Y\in \mathscr{M}_{k\times \ell}\mid \rank(Y_{r\times s})\ls \rank(w_{r\times s}), \text{  for all }1\ls r\ls k, 1\ls s\ls \ell
\right\}.
$$
\item The defining ideal of $\overline{X_w}$ is the {\it Schubert determinantal ideal} $I_w$ which is generated by the $\rank(w_{r\times s})+1$ minors of $X_{r\times s}$ for every $1\ls r\ls k, 1\ls s\ls \ell$.
\end{enumerate}	

\end{definition}

\begin{remark}
	\
	\begin{enumerate}[(1)]
		\item Given $1\ls t\ls \min\{k,\ell\}$, if $w\in \mathscr{M}_{k\times \ell}(\mathbb{\kk})$ is the partial permutation defined by  $w_{ii}=1$ for $1\ls i\ls t-1$, then $I_w$ is the classical determinantal ideal of $t$-minors of $X$. 
			\item 
		Mixed ladder determinantal ideals of  ladders with only one upper-outside corner coincide with the Schubert determinantal ideals of {\it vexillary permutations} \cite[Proposition 9.6]{fulton1992flags}.
		\item For any partial permutation  there are smaller sets of generators for $I_w$ (see \cite[Theorem 15.15]{MillerSturmfels}, \cite{gao2022minimal}).

	   
    \end{enumerate}
\end{remark}
%
%


We continue with the definition of the following class of ideals.
\begin{definition}
	An {\it  ideal} of $\Pi$ is a subset $\Omega\subseteq \Pi$ such that $$\text{ for all } \omega\in \Omega \text{ and } \pi\in\Pi, \text{  if } \pi\leq \omega \text{ then } \pi\in\Omega.$$  
For any  ideal $\Omega$ of $\Pi$, we call the ideal generated by $\Omega$ in $\kk[X]$  a {\it poset ideal}.
\end{definition}

%

\begin{remark}

For fixed $\delta\in\Pi$,
a particular type of ideal is the set
 $$\Omega_\delta:=\{\pi\in\Pi: \pi\not\gs \delta\}.$$
The poset ideals $\Omega_\delta\kk[X]$ are 
the defining ideals of 
varieties 
originating as certain affine charts of Schubert varieties  \cite[p. 535]{procesi2007lie} (see also \cite[Chapter 5, \S A]{BookDet}).
\end{remark}

\section{Chamfering ladders}\label{Decorn_Section}
	
In this section we define a combinatorial operation on ladders, which we call  {\it chamfering}, and show some important consequences of it. This definition is inspired by the work of previous authors \cite{conca1995ladder, concaherzog1997, GorlaMixed, gonciulea2000mixed}.

\begin{setup} \label{setup_chamfer} 
	We adopt the assumptions and notations from \autoref{def_ladder}, \autoref{def_mixed_lad_det}, and \autoref{running assumptions ladder}. 
	In particular, $X=(x_{i,j})$ is a generic matrix of size $k \times \ell$ and $L\subseteq	 X$ is a ladder. Let $\kk$ be a field. We denote by $R$ the polynomial ring $R=\kk[L]$. Given   $\ttt=(t_1,\ldots,t_v) \in \ZZ_{>0}^v$ we denote by $I_\ttt(L)$ the corresponding mixed ladder determinantal ideal.    
\end{setup}

\begin{definition}\label{def_chamfer}
Assume \autoref{setup_chamfer}. We say a ladder $L$ is {\it non-degenerate} with respect to $\ttt$ if $I_\ttt(L)$ satisfies the assumptions in \autoref{running assumptions ladder}.  
We call the quantity $$\sum_{i<j}|t_i-t_j|$$
the {\it total width} of $L$ with respect to $\ttt$.   
Let  $\left\{x_{d_1,1}=x_{d_1,c_1},\ldots,x_{d_v,c_v}=x_{k,c_v}\right\}$ be the lower-outside corners of $L$. We say that a pair $(L', \ttt')$ with $L'\subset X$ a ladder and $\ttt'\in \NN^d$  is the {\it chamfer} of $(L,\ttt)$ at $x_{d_j,c_j}$ (or that $(L', \ttt')$ is obtained from $(L,\ttt)$ by {\it chamfering} $x_{d_j,c_j}$) if $\ttt'=(t_1,\ldots, t_{j-1}, t_j-1,t_{j+1},\ldots, t_v)$, and $L'$ has the same upper-outside corners of $L$ and lower-outside corners $$\left\{x_{d_1,1}=x_{d_1,c_1},\ldots,x_{d_{j-1},c_{j-1}},\, x_{d_j-1,c_j+1},\, x_{d_{j+1},c_{j+1}},\ldots,x_{d_v,c_v}=x_{k,c_v}\right\}$$ for some  
$x_{d_j,c_j}$ such that $d_j<d_{j-1}$ and $c_j>c_{j-1}$. 
\end{definition}


\begin{proposition} \label{prop_chamfer}
Assume \autoref{setup_chamfer}, and let $L$ be a non-degenerate ladder with respect to $\ttt$. Then
\begin{enumerate}[{\rm (1)}]
	\item if $L'$ is the chamfer of $(L,\ttt)$ at $x_{d_j,c_j}$, then  $L'$  is non-degenerate.
	\item $L$ is the chamfer of a non-degenerate ladder with strictly smaller total width. 
\end{enumerate}
\end{proposition}
\begin{proof} We note that condition (3) of \autoref{running assumptions ladder} holds if and only if $d_{j+1}-d_j>t_{j+1}-t_j$ and $c_{j+1}-c_j>t_{j}-t_{j+1}$, and one can readily check that it holds for $L'$. The fact that $L'$ satisfies the other conditions of \autoref{running assumptions ladder} is an easy check. For the second claim, let $j$ be such that $t_j=\min\{t_1,\ldots, t_v\}$. Consider the ladder $L''$ with same upper-outside corners as $L$ and lower outside corners $$\left\{x_{d_1,1}=x_{d_1,c_1},\ldots,x_{d_{j-1},c_{j-1}},\, x_{d_j+1,c_j-1},\, x_{d_{j+1},c_{j+1}},\ldots,x_{d_v,c_v}=x_{k,c_v}\right\}.$$
We note that $L''$ is non-degenerate with respect to $\ttt''=(t_1,\ldots, t_{j-1}, t_j+1,t_{j+1},\ldots, t_v)$, $(L,\ttt)$ is the chamfer of $(L'',\ttt'')$ at $x_{d_j+1,c_j-1}$, and the total width of $L''$ is strictly smaller than $L$.
\end{proof}

\begin{corollary} \label{coroll_chamfer}
Assume \autoref{setup_chamfer}. If $L$ is  non-degenerate with respect to $\ttt$, then there exists a ladder $L''$ and $t,u\in \NN$ such that $L''$ is non-degenerate with respect to $(t,\ldots, t)\in \NN^u$ and $L$ is obtained from $L''$ by a finite sequence of chamferings. 
\end{corollary}
\begin{proof}
We proceed by descending induction on the total width of the ladder, making a repeated use of \autoref{prop_chamfer} (2).
\end{proof}

For an ideal $I$ in a commutative ring with identity $A$, we denote the {\it symbolic Rees algebra} of $I$ by $\R^s(I)=\oplus_{n\gs 0}I^{(n)}T^n\subseteq A[T]$, where $T$ is a variable. Sometimes we write  $\R^s_A(I):=\R^s(I)$ if the ambient ring needs to be specified. We denote by $\gr^s(I)=\oplus_{n\gs 0}I^{(n)}/I^{(n+1)}$ the {\it symbolic associated graded algebra} of $I$. 

\begin{theorem} \label{ThmNoetherian}
Assume \autoref{setup_chamfer}, and let $L$ be a non-degenerate ladder with respect to $\ttt$. Then $\R^s(I_\ttt(L))$ and $\gr^s(I_\ttt(L))$ are Noetherian rings.
\end{theorem}
\begin{proof}
It suffices to show the statement for $\R^s(I_\ttt(L))$. By Corollary \ref{coroll_chamfer} there exists a ladder $L''$ which is non-degenerate with respect to $(t,\ldots,t) \in \NN^u$, and $L$ is obtained from $L''$ by a finite sequence of chamferings. Then by Gorla's work on mixed ladder determinantal ideals \cite[Lemma 1.19]{GorlaMixed} there are indeterminates $\{y_1,\ldots,y_d\}$ of $R''=\kk[L'']$ and $c \leq d$ such that 
\[
I_\ttt(L)[y_1,\ldots,y_d,y_1^{-1},\ldots y_c^{-1}] \cong I_t(L'')[y_1^{-1},\ldots y_c^{-1}].
\]
From this isomorphism we conclude that for all $n \geq 1$
\[
I_\ttt(L)^{(n)}[y_1,\ldots,y_d,y_1^{-1},\ldots y_c^{-1}] \cong I_t(L'')^{(n)}[y_1^{-1},\ldots y_c^{-1}].
\]
Since $\R^s_{R''}(I_t(L''))$ is Noetherian \cite{brunsconca98}, so is $\R^s_R(I_\ttt(L))[y_1,\ldots,y_d,y_1^{-1},\ldots y_c^{-1}]$. As $\R^s_R(I_\ttt(L))$ is a direct summand of the latter, we conclude that it is itself Noetherian.
\end{proof}

\begin{remark} Unmixed ladder determinantal ideals define $F$-rational singularities in characteristic $p\gg0$, and rational singularities in characteristic zero, by work of Conca and Herzog \cite{concaherzog1997}. It follows from \autoref{coroll_chamfer} that this is also true for mixed ladder determinantal ideals since being $F$-rational or having rational singularities localizes, and because if $S$ is a ring and $x$ is an indeterminate, one has that $S$ is $F$-rational (resp. has rational singularities) if and only if $S[x]$ is $F$-rational (resp. has rational singularities).
\end{remark}

\section{Mixed ladder determinantal ideals}\label{MLDSection}

In this section we prove the main results of this paper regarding the $F$-singularities of ladder determinantal rings and their blowup algebras. Throughout we use the following setup:


\begin{setup} \label{setup_mainsec_ladder} 
We adopt \autoref{setup_chamfer}, and when the base field $\kk$ has characteristic $p>0$ we further assume that it is $F$-finite. For  $r=2,\ldots ,k+\ell$ we denote by $\mathcal{D}_r$ the antidiagonal at level $r$ of $X$, namely, 
\[\mathcal{D}_r=\{x_{i,j}\in X\mid i+j=r\}\subseteq X.\]
We define $L^\circ_{\ttt} \subseteq L$ to be the subladder with the same upper-outside corners as $L$, and with lower-outside corners  $\{x_{d_j-t_j+1,c_j+t_j-1} \mid 1 \leq j \leq v\}$. Notice that $L^\circ_{\ttt} =\bigcup_{i=1}^v (L_j)^\circ_{t_j}$ (see \autoref{eq_Lj}).
\end{setup}

The following is the first main result of this section. 

\begin{theorem}\label{ThmMixLadderSFP}
Assume \autoref{setup_mainsec_ladder} with   $\mathbb{k}$ of  characteristic $p>0$. 
Then $I_{\ttt}(L)$ is symbolic $F$-split.
\end{theorem}
\begin{proof}
Let $h = \Ht\left(I_{\ttt}(L)\right)$. It suffices to find an element $f \in I_{\ttt}(L)^{(h)}$ and a monomial order $\prec$ such that $\IN_\prec(f)$ is squarefree \cite[Lemma 6.2(1)]{DSMNB}. 
Let 
\begin{equation}\label{eq_inv1}
\mathcal{A} = \{2 \leq r \leq k+\ell \mid \mathcal{D}_r\cap L^\circ_\ttt\neq\emptyset\},
\end{equation}
and for every $r \in \mathcal{A}$  set 
\begin{equation}\label{eq_inv2}
w_r=\max\{w\mid x_{w,r-w}\in \mathcal{D}_r\cap L^\circ_\ttt\}
\qquad
\text{and}
\qquad
p_r = \max\{j \mid x_{w_r, r-w_r}\in  (L_j)^\circ_{t_j} \}.
\end{equation}
Furthermore,  set 
\begin{equation}\label{eq_inv3}
a_r=\min\left\{i\mid x_{i,r-i}\in L_{p_r}\cap \mathcal{D}_r\right\} \qquad \text{and}\qquad b_r=\max\left\{i\mid x_{i,r-i}\in L_{p_r}\cap \mathcal{D}_r\right\}.
\end{equation}
Let $\gamma_r= b_r-a_r+1$ and observe that the $\gamma_r \times \gamma_r$ matrix $Y_r=X^{[a_r,b_r]}_{[r-b_r,r-a_r]}$ is contained in $L_{p_r}$ (see \autoref{notation_affine_subsec}). 
  We note that 
\begin{equation}\label{eq_detY}
\det(Y_r) \in I_{\gamma_r}(Y_r) \subseteq I_{t_{p_r}}(Y_r)^{\left(\gamma_r-t_{p_r}+1\right)} \subseteq I_{t_{p_r}}(L_{p_r})^{\left(\gamma_r-t_{p_r}+1\right)} \subseteq I_{\ttt}(L) ^{\left(\gamma_r-t_{p_r}+1\right)},	
\end{equation}
where the first inclusion follows from \cite[Proposition 10.2]{BookDet}. Set 
\begin{equation}\label{eq_f}
f=\prod_{r \in \mathcal{A}} \det(Y_r)
\end{equation}
 We claim that $f \in I_{\ttt}(L)^{(h)}$. Observe that for every $r \in \mathcal{A}$, the positive integer $\gamma_r-t_{p_r}+1$ counts  the number of elements on the main antidiagonal of $Y_r$ that belong to $L^\circ_\ttt$. 
Moreover, 
every entry of $L^\circ_\ttt$  belongs to the main antidiagonal of  $Y_r$ for exactly one $r \in \mathcal{A}$. Thus, by  \cite[Theorem 1.15]{GorlaMixed} we  have 
\begin{equation}\label{eq_sum_height}
	\sum_{r \in \mathcal{A}}(\gamma_r-t_{p_r} + 1) = |L^\circ_\ttt| = h.
\end{equation}
 Therefore, \autoref{eq_detY} implies
\[
f \in \prod_{r \in \mathcal{A}} I_{\ttt}(L)^{\left(\gamma_r-t_{p_r}+1\right)} \subseteq I_{\ttt}(L)^{(h)},
\]
whence the claim follows. Finally, if $\prec$ is any antidiagonal term order, then $\IN_\prec(f)$ is squarefree, completing the proof.
\end{proof}

As a consequence of \autoref{ThmMixLadderSFP} we 
obtain that symbolic blowup algebras of mixed ladder determinantal ideals, and hence also rings defined by such ideals, are $F$-pure.

\begin{corollary}\label{blowups_f_pure}
Assume \autoref{setup_mainsec_ladder} with   $\mathbb{k}$ of  characteristic $p>0$.
Then the blowup algebras $\R^s\left(I_{\ttt}(L)\right)$ and $\gr^s\left(I_{\ttt}(L)\right)$ are $F$-pure. In particular, $R/I_\ttt(L)$ is $F$-pure.
\end{corollary}
\begin{proof}
The first claim follows from  \autoref{ThmMixLadderSFP} and properties of symbolic $F$-split ideals \cite[Theorem 4.7]{DSMNB}. The second follows from the fact that $R/I_\ttt(L)$ is a direct summand of $\gr^s(I_\ttt(L))$.
\end{proof}

Let $S=\kk[x_1,\ldots, x_d]$ be a standard graded polynomial ring, and $\m=(x_1,\ldots,x_d)$. For a finitely generated $\ZZ$-graded $S$-module $M$ we  denote by $\reg(M)$ its   {\it Castelnuovo-Mumford regularity}, i.e., $\sup\{j+i \mid H^i_\m(M)_j \ne 0\}$.

\begin{corollary}\label{cor_RegDEpth}
Assume \autoref{setup_mainsec_ladder} with   $\mathbb{k}$ of  characteristic $p>0$.  Then,
$\lim\limits_{n\to\infty}\frac{\reg\left(R/I_\ttt(L)^{(n)}\right)}{n}$ exists and the sequence
$\{\Depth(R/I_\ttt(L)^{(n)})\}_{n \in \NN}$ is constant for $n\gg0$, with value equal to $\min\{\depth(R/I_\ttt(L)^{(n)})\}_{n \in \NN}$.
\end{corollary}
\begin{proof}
The statements follow from  \autoref{ThmMixLadderSFP} and properties of symbolic $F$-split ideals \cite[Theorem 4.10]{DSMNB}, using the fact that $\cR^s(I_\ttt(L))$ is Noetherian by \autoref{ThmNoetherian}. 
\end{proof}

%

We also show that symbolic powers commute with taking initial ideals for unmixed ladder determinantal ideals.
\begin{corollary}\label{cor_SAGBI}
Assume \autoref{setup_mainsec_ladder} with   $\mathbb{k}$ a perfect field of  characteristic $p>0$, and let $I=I_\ttt(L)$. 
Let $\prec$ be an antidiagonal term order, and assume that $\IN_\prec\left(\cR^s(I)\right) := \bigoplus_{n \geq 0} \IN_\prec\left(I^{(n)}\right)$ is Noetherian. If $p \gg 0$, then we have
\[
\IN_\prec\left(I^{(n)}\right)=\left(\IN_\prec \left(I\right)\right)^{(n)} \,\,\, \text{ for every }n\in\NN.
\]
In particular, if $I=I_t(L)$ for some $t \in \ZZ_{>0}$, then the above equality holds for all $p>(\min\{k,\ell\}-t+1)!$.
\end{corollary}
\begin{proof}
This follows at once from  \autoref{symbolic-initial}. For the case $I=I_t(L)$, the bound for $p$ follows from \autoref{remark_char} and the fact that $\IN_\prec(\cR^s(I))$ is generated as an $R$-algebra by forms of degree at most $\min\{k,\ell\}-t+1$  \cite[Theorem 4.1]{brunsconca98}. 
\end{proof}


The following technical lemma is needed for the proof of Theorem \ref{stronglyF}.

\begin{lemma}\label{Lemma sfr} Let $B$ be an $F$-finite ring of characteristic $p>0$, $J \subseteq B$ an ideal, and $A=B[x_1,\ldots,x_d]$ a polynomial extension. Let $H=(x_1,\ldots,x_d)$ and  $I=JA+H$. Then $\cR^s_A(I)$ is strongly $F$-regular if and only if $\cR^s_B(J)$ is strongly $F$-regular.
\end{lemma}
\begin{proof}
Since $H^{(j)}=H^j$ for all $j \in \ZZ_{>0}$,   we have that $I^{(j)} = \sum_{i=0}^j (JA)^{(i)}H^{j-i}$  \cite[Theorem 3.4]{BinomialSumSP}. Thus
\[
\cR^s_A(I) = \cR^s_A(JA)[x_1T,\ldots,x_dT]=\cR^s_B(J)[x_1,\ldots,x_d,x_1T,\ldots,x_dT],
\]
where the second equality holds as   $(JA)^{(i)} = J^{(i)}A = J^{(i)}[x_1,\ldots,x_d]$. Thus $\cR^s_A(I)$ is isomorphic to $\cR^s_B(J)[x_1Z,\ldots,x_dZ,x_1T,\ldots,x_dT]$, with $Z$ another variable, and so isomorphic to the Segre product  $\cR^s_B(J)[x_1,\ldots,x_d]\,
{\mathlarger{\mathlarger\#}}\,\cR^s_B(J)[Z,T]$. 

Now, assume  $\cR^s_A(I)$ is strongly $F$-regular. Since $\cR^s_B(J)$ is a direct summand of $\cR^s_A(I)$, it is strongly $F$-regular  \cite[Theorem 3.1(e)]{HHStrongFreg}. Conversely, if $\cR^s_B(J)$ is strongly $F$-regular, then so is $\cR^s_A(I)$ since the previous description as a Segre product realizes it as a direct summand of the polynomial extension $\cR^s_B(J)[x_1,\ldots,x_d,Z,T]$.
\end{proof}

The following is one of the main results of this paper. We show that the symbolic Rees algebra of any mixed ladder determinantal ideal is strongly $F$-regular. 

\begin{theorem}\label{stronglyF}
Assume \autoref{setup_mainsec_ladder} with   $\mathbb{k}$ of characteristic $p>0$. Then  $\R^s\left(I_{\ttt}(L)\right)$ is strongly $F$-regular.
\end{theorem}
\begin{proof}

We proceed by induction on $d=\dim(R)$, i.e., on the number of entries of $L$. The base case $d=1$ is trivial.

If $t_v=1$, then we can write $I_{\ttt}(L)=I_{\ttt'}(L') + I_1(L_v)$ where $L'$ is a proper subladder of $L$ and $\ttt' = (t_1,\ldots,t_{v-1})$. 
Let $R' = \mathbb{k}[L']$. Since $\dim(R')<\dim(R)$, by the induction hypothesis we have that $\cR^s_{R'}\left(I_{\ttt'}(L')\right)$ is strongly $F$-regular. It follows that $\cR^s_R\left(I_{\ttt}(L)\right)$ is strongly $F$-regular by \autoref{Lemma sfr}.

Now assume $t_v>1$ and  consider the smallest $\alpha \in \{1,\ldots,v\}$ for which $(k,c_\alpha)$ is a lower-outside corner of $L$. We observe that such an $\alpha$ exists since $(k,c_v)$ is always a lower-outside corner by our  assumptions (see Notation \ref{running assumptions ladder}). 
Let $\mathcal{A},p_r, a_r, b_r$ be as in \autoref{eq_inv1}, \autoref{eq_inv2}, and \autoref{eq_inv3}. For every $r\in\mathcal{A}$ set $\gamma_r= b_r-a_r+1$ and  $Y_r=X^{[a_r,b_r]}_{[r-b_r,r-a_r]}$. 
 By \autoref{eq_detY} we have $\det(Y_r) \in I_\ttt(L)^{\left(\gamma_r-t_{p_r}+1\right)}$.  
 Note that $\beta:=k+\gamma_\alpha \in \mathcal{A}$, $b_\beta = k$,  and $p_\beta = v$. 
 Set $Y=X^{[a_{\beta},k-1]}_{[c_\alpha+1,\beta-a_{\beta}]}$ and observe that 
$ \det(Y) \in I_{\gamma_\beta-1}(Y) \subseteq I_{t_v}(Y)^{(\gamma_\beta-t_v)}\subseteq  I_\ttt(L)^{(\gamma_\beta-t_v)}$  \cite[Proposition 10.2]{BookDet}.
Set $$
\mathcal{A}\,' = \mathcal{A} \smallsetminus \{\beta\}
\qquad 
\text{and} 
\qquad
g=\det(Y)  \prod_{r \in \mathcal{A}\,'} \det(Y_r).$$ We claim that $g \in I_\ttt(L)^{(h-1)}$, where $h=\Ht\left(I_\ttt(L)\right)$. 
For every $r \in \mathcal{A}$, the positive integer $\gamma_r-t_{p_r}+1$ counts  the number of elements on the main antidiagonal of $Y_r$ that belong to $L^\circ_\ttt$. By \autoref{eq_sum_height} we  have
 $$(\gamma_{\beta}-t_v)+\sum_{r \in \mathcal{A}\,'}(\gamma_r-t_{p_r} + 1) =\sum_{r \in \mathcal{A}}(\gamma_r-t_{p_r} + 1)=  h-1.$$ 
 Therefore $g \in I_\ttt(L)^{(\gamma_{\beta}-t_v)}  \prod_{r \in \mathcal{A}\,'} I_\ttt(L)^{(\gamma_r-t_{p_r}+1)} \subseteq I_\ttt(L)^{(h-1)}$, which proves the claim.

  Now let $\prec$ be any antidiagonal term order and note that  $\IN_\prec(g)$ is a squarefree monomial not divisible by $x_{k,c_\alpha}$. Moreover,  we have $g^{p-1}\in \left(I_{\ttt}(L)^{(n)}\right)^{[p]}:I_{\ttt}(L)^{(np)}$ for every $n\in \NN$ \cite[Lemma 2.6]{GrifoHuneke}. Thus there exists an $\cR^s_R\left(I_{\ttt}(L)\right)$-homomorphism
  $\Psi:\cR^s_R\left(I_{\ttt}(L)\right)^{1/p}\to\cR^s_R\left(I_{\ttt}(L)\right)$ such that 
 $\Psi\left(x_{k,c_\alpha}^{1/p}\right)=1$ (cf. proof of \cite[Theorem 6.7]{DSMNB}).

Since $\R^s(I_\ttt(L))$ is Noetherian by \autoref{ThmNoetherian}, and $\kk$ is $F$-finite, we have that $\R^s(I_\ttt(L))$ is $F$-finite. By  \cite[Theorem 3.3]{HHStrongFreg}, to finish the proof it suffices to show that $\cR^s_R(I_{\ttt}(L))_{x_{k,c_\alpha}}$ is strongly $F$-regular. Let $L''$ be the ladder obtained from $L$ by deleting the entries $$E=\{x_{d_{\alpha-1}+1,c_\alpha},x_{d_{\alpha-1}+2,c_\alpha},\ldots,x_{k,c_\alpha},x_{k,c_\alpha+1},\ldots,x_{k,c_{\alpha+1}-1}\}$$ from $L$. Set $R''=\mathbb{k}[L'']$ and $S=R''[E]_{x_{k,c_\alpha}}$.  There is an isomorphism 
  $$\theta:R_{x_{k,c_\alpha}} \to S 
\quad \text{such that} 
\quad 
\theta\left(I_\ttt(L)_{x_{k,c_\alpha}}\right) = I_{\ttt''}(L'')S,
  $$ 
  where $\ttt''=(t_1,\ldots,t_{\alpha -1},t_\alpha -1,t_{\alpha +1}, t_{v-1},t_v)$ and $I_{\ttt''}(L'')$ is computed in $R''$  \cite[Lemma 1.19]{GorlaMixed}. One can check that  $\theta\left(I_\ttt(L)^{(n)}_{x_{k,c_\alpha}}\right) = I_{\ttt''}\left(L''\right)^{(n)}S$ for all $n \geq 1$ (see \cite[Lemma 10.1]{BookDet} for an analogous procedure). By the induction hypothesis, $\cR^s_{R''}(I_{\ttt''}(L''))$ is strongly $F$-regular, and therefore $\cR^s_{R''}(I_{\ttt''}(L'')) \otimes_{R''} S$ is strongly $F$-regular. Therefore we conclude via the  isomorphism $\theta$ that $\cR^s_R(I_{\ttt}(L)) \otimes_R R_{x_{k,c_\alpha}} \cong \cR^s_R(I_{\ttt}(L))_{x_{k,c_\alpha}}$ is strongly $F$-regular, finishing the proof.
\end{proof}

The rest of the this section is devoted to proving the following theorem which states  that unmixed determinantal ideals of two-sided ladders are Knutson ideals. Following \autoref{notation_affine_subsec}, for a ladder $L\subset X$ we denote by  $L_{[a,b]}=X_{[a,b]} \cap L$ the  subladder corresponding to the  consecutive columns from $a$ to $b$. Likewise,  $L^{[c,d]}=X^{[c,d]} \cap L$ denotes the subladder corresponding to the  consecutive rows from $c$ to $d$. 

\begin{theorem}\label{ThmLadderKI}
Assume \autoref{setup_mainsec_ladder}. For $t \in \ZZ_{>0}$  we let $I_t(L) \subseteq R=\mathbb{k}[L]$ be the corresponding unmixed ladder determinantal ideal. Let 
 $f$ be as in \autoref{eq_f} with $t_1=\cdots =t_v=t$. Then $$I_t\left(L_{[a,b]}\right) \in \mathcal{C}_f \qquad 
\text{ for every } \qquad 1\leq a < b\leq \ell \quad \text{ such that } \quad I_t\left(L_{[a,b]}\right)\neq 0,\qquad \text{and}$$   $$I_t\left(L^{[c,d]}\right) \in \mathcal{C}_f \qquad  
\text{ for every } \qquad 1\leq c < d\leq k \quad \text{ such that } \quad I_t\left(L^{[c,d]}\right)\neq 0.$$
In particular, then $ I_t(L)\in \mathcal{C}_f$ whenever $I_t(L)$ is not zero.
\end{theorem}

 A first step towards the proof of  \autoref{ThmLadderKI} is the following lemma where we show that the ideal generated by the $t$-minors of a subladder corresponding to $t$ consecutive columns or rows is in $\mathcal{C}_f$.
\begin{lemma}\label{lemmaLadderKI}
Under the assumptions and notations of \autoref{ThmLadderKI} we have
$$
 I_t\left(L_{[j,j+t-1]}\right) \in \mathcal{C}_f \qquad  \text{ for every } \,\,\,\, 1\ls j\ls  \ell-t+1 \quad \text{ such that } \quad I_t\left(L_{[j,j+t-1]}\right)\neq 0,\qquad \text{and}
 $$
 $$
 I_t\left(L^{[i,i+t-1]}\right) \in \mathcal{C}_f \qquad\text{ for every } \qquad 1\ls i\ls k-t+1 \quad \text{ such that } \quad I_t\left(L^{[i,i+t-1]}\right)\neq 0.
 $$
\end{lemma}

\begin{proof}
Fix $1\ls j\ls \ell-t+1$ such that $I_t\left(L_{[j,j+t-1]}\right)\neq 0$. Set $L'=L_{[j,j+t-1]}$ and  $\mathcal{E}=\{2 \leq r \leq k+\ell \mid \left|\mathcal{D}_{r}\cap L'\right|=t \}$. 
Notice that $\mathcal{E} \neq 0$, and for every $r \in \mathcal{E}$ there is only  one element on the main antidiagonal of $Y_{r}$ that belongs to $(L')^\circ_t$ (see sentence after \autoref{eq_inv3}). Then by  \cite[Theorem 1.15]{GorlaMixed} it follows that $\Ht \left(I_t(L')\right)=|(L)^\circ_t| = | \mathcal{E}|$. Moreover
\begin{equation*}\label{eq_dets_in}
	H:=\left\langle \left\{\det(Y_{r}) \mid r \in \mathcal{E}\right\}\right\rangle \subseteq I_t(L')  . 
\end{equation*}
Notice $H$ is a complete intersection with $\Ht (H)=|\mathcal{E}|$ and $H\in \mathcal{C}_f$ as it is generated by some of the factors of $f$. Thus $I_t(L')$ is a minimal prime of $H$ and so $I_t(L')\in \mathcal{C}_f$. 

Likewise, one can prove that $I_t\left(L^{[i,i+t-1]}\right)$ is either zero or in $ \mathcal{C}_f$ for every $i$.
\end{proof}

With \autoref{lemmaLadderKI} in hand we can  prove  \autoref{ThmLadderKI}.

\begin{proof}[Proof of \autoref{ThmLadderKI}]
We will only show that $I_t\left(L_{[a,b]}\right) \in \mathcal{C}_f$ whenever $I_t\left(L_{[a,b]}\right) \neq 0$, as the proof for $I_t\left(L^{[c,d]}\right) $ is similar.

We proceed by induction on $\delta\in \NN$ to show that $I_t\left(L_{[j,j+t-1+\delta]}\right)$ is either 0 or in $\mathcal{C}_f$  for every $1\ls j\ls \ell-t+1-\delta$, the base case $\delta=0$ being 
\autoref{lemmaLadderKI}. 

Fix $1\ls j\ls \ell-t-\delta$. By  the induction hypothesis  the ideals $I_t\left(L_{[j,j+t-1+\delta]}\right)$ and $I_t\left(L_{[j+1,j+t+\delta]}\right)$ are either zero or in $\mathcal{C}_f$. Therefore so is their  sum. Assume that the sum is not zero.  

We claim that 
\begin{equation}\label{eq_first_int}
I_t\left(L_{[j,j+t-1+\delta]}\right)+I_t\left(L_{[j+1,j+t+\delta]}\right)=I_t\left(L_{[j,j+t+\delta]}\right) \cap I_{t-1}\left(L_{[j+1,j+t-1+\delta]}\right),
\end{equation}
if  $t>1$, 
 and  
\begin{equation}\label{eq_second_int}
	I_t\left(L_{[j,j+t-1+\delta]}\right)+I_t\left(L_{[j+1,j+t+\delta]}\right)=I_t\left(L_{[j,j+t+\delta]}\right),
\end{equation}
otherwise.

Since \autoref{eq_second_int} is clear, we focus on proving \autoref{eq_first_int}. So, we may assume  $t>1$.  
To simplify the notation we set 
$$L':=L_{[j,j+t-1+\delta]}\qquad \text{and} \qquad L'':=L_{[j+1,j+t+\delta]}.$$ 
 From the proof of \cite[Theorem~2.1]{seccia2022knutson} it follows that $$I_t\left(X_{[j,j+t-1+\delta]}\right)+I_t\left(X_{[j+1,j+t+\delta]}\right)=I_t\left(X_{[j,j+t+\delta]}\right) \cap I_{t-1}\left(X_{[j+1,j+t-1+\delta]}\right).$$
 Thus after contracting these ideals to $\mathbb{k}[L' \cup L'']$ we obtain
$$
  \left( I_t\left(X_{[j,j+t-1+\delta]}\right)+I_t\left(X_{[j+1,j+t+\delta]}\right)\right) \cap \mathbb{k}[L' \cup L'']=I_t(L') + I_{t}(L''),\qquad\text{and}$$
  $$
  \left(I_t\left(X_{[j,j+t+\delta]}\right) \cap I_{t-1}\left(X_{[j+1,j+t-1+\delta]}\right)\right)\cap \mathbb{k}[L' \cup L'']= I_t(L' \cup L'') \cap I_{t-1}(L' \cap L''),
$$
where the first equality follows from \cite[Lemma 2.2]{seccia2022knutson}, \cite[Corollary 2]{knutson2009frobenius}, and  \cite[Lemma 1.8]{GorlaMixed}, whereas the second one is clear.  
 This proves the claim. 
 
 Hence the ideal
%
%
$I_t(L_{[j,j+t+\delta]})=I_t(L' \cup L'')$ is a minimal prime of $I_t(L')+I_t(L'') \in \mathcal{C}_f$. We conclude  that $I_t(L_{[j,j+t+\delta]})$ is either 0 or in $\mathcal{C}_f$ for every $1\ls j\ls  \ell-t$, finishing the proof.  
\end{proof}

\section{Knutson  Schubert determinantal and poset ideals}\label{SectionSchubertMSV}

In this section, we show that Schubert determinantal ideals and poset determinantal ideals are Knutson ideals. We remark that the statement for  Schubert determinantal ideals  already appeared in Knutson's work 
 \cite{knutson2009frobenius}. 

Throughout this section we follow the notations introduced in \autoref{sub_affine_schub}. In particular, $X=(x_{i,j})$ is a generic matrix of size $k \times \ell$. 
We also assume the following setup.

\begin{setup}\label{setup_Knutson_section}
	We adopt  \autoref{notation_affine_subsec}. 
	Let $R=\kk[X]$ be the polynomial ring in the variables in $X$, where $\kk$ is an arbitrary field. Moreover, we equip $R$ with an antidiagonal term order.
\end{setup}

We continue with the definition of the following polynomial which plays an essential role in the results in this section. This polynomial is inspired by the results by the fourth-named author \cite{seccia2022knutson}.
\begin{definition}\label{def_f(X)}
	Assume  \autoref{setup_Knutson_section}. We define $f(X) \in \kk[X]$ as: 
	$$f(X)= \prod_{j=0}^{k-2} \left[\det \left(X_{[1,j+1]}^{[1,j+1]}\right) 
	\det \left(X^{[k-j,k]}_{[\ell-j,\ell]}\right) \right] 
	\prod_{j=1}^{\ell-k+1} \left(\det \left(X_{[j,k+j-1]}\right) \right).$$ 
	We note that: 
	$$\init (f(X))= \prod_{\substack{1\ls i\ls k \\ 1\ls j\ls \ell}} x_{i,j}.$$
\end{definition}

The following theorem shows that determinantal ideals of every northwest and southeast submatrix of $X$ is Knutson associated to $f(X)$. This allows us to show Schubert determinantal and poset ideals are Knutson.

\begin{theorem}\label{p:msk}
	Assume \autoref{setup_Knutson_section} and  let $f:=f(X) $ be as in \autoref{def_f(X)}.  Then for every $1\ls r \ls k$, $1\ls s\ls \ell$, and $t\ls \min\{r,s\}$, we have $I_t(X_{r\times s})\in \mathcal{C}_f$ and $I_t(X^{r\times s}) \in \mathcal{C}_f$.
\end{theorem}

\begin{proof}
	If $r=k$, or $s=\ell$, then $I_t(X_{r\times s}) \in \mathcal{C}_f$ \cite[Lemma 2.2]{seccia2022knutson}. Hence, we can assume   $r\neq k$ and $s \neq \ell$. We set $m:=\min\{r,s\}$ and $M:=\max\{r,s\}$.  
	We proceed by induction on $\delta:=m-t$.
	
	{\bf Base case:} Assume $\delta=0$. We define, 
	$$D_1:=\left( \deta{X_{m \times m}}, \ldots,\deta{X_{M \times M}}\right)\text{ if }M\ls k, \text{and}$$
	$$D_1:=\left( \deta{X_{m \times m}}, \ldots,\deta{X_{k \times k}}, \deta{X_{[2,k+1]}},\ldots, \deta{X_{[M-k+1,M]}} \right) \text{ otherwise}.$$
The ideal $D_1$ is a complete intersection of  height  $M-m+1=\height\left(I_m(X_{r \times s})\right)$ and it  is contained in $I_m(X_{r \times s})$, therefore $I_m(X_{r \times s}) \in \MIN( D_1).$  
Moreover, $D_1 \in \mathcal{C}_f$ because it is generated by factors of $f$ and so it is a sum of minimal primes of $(f)$.  Thus,  $I_m(X_{r \times s})\in \mathcal{C}_f$ for every $r$ and $s$.

{\bf Induction step:}  Assume $\delta\gs 1$. We have the inclusion
\begin{equation}\label{cont1}
	I_{t}(X_{r \times s}) \supseteq D_2:=I_{t}(X_{r-1 \times s}) + I_{t}(X_{r \times s-1}) +I_{t+2}(X_{r+1 \times s+1}).
\end{equation}
The ideal  $J:= I_{t}(X_{r-1 \times s}) + I_{t}(X_{r \times s-1})$ is  ladder determinantal  (see  \autoref{sub_ladder}), so it is prime of height $(r-t+1)(s-t+1)-1$ \cite{GorlaMixed}. Since $J$ does not contain $I_{t+2}(X_{r+1 \times s+1})$, we obtain that 
$$(r-t+1)(s-t+1)=\height\left(I_{t}(X_{r \times s})\right)\gs\height(D_2)\gs \height(J)+1\gs (r-t+1)(s-t+1),$$ which implies 
\begin{equation}\label{inMin}
	I_{t}(X_{r \times s}) \in \MIN ( D_2).
\end{equation}
Now, assume $s\gs r$, i.e., $m=r$. By the $\delta -1$ case of the induction, $I_{t}(X_{r-1 \times s})$ and $I_{t+2}(X_{r+1 \times s+1})$ belong to $\mathcal{C}_f$. 
If $r=s$, we have $I_{t}(X_{r \times s-1})\in \mathcal{C}_f$ again  by   the $\delta -1$ case of the induction, and then, $I_{t}(X_{r \times s})\in \mathcal{C}_f$ by \eqref{cont1} and \eqref{inMin}. Now,  using \eqref{cont1} and \eqref{inMin} repeatedly we conclude  $I_{t}(X_{r \times s}) \in \mathcal{C}_f$ for any $s\gs r$. The case $r\gs s$ follows similarly by reversing the roles of $r$ and $s$.  Thus, $I_t(X_{r\times s}) \in \mathcal{C}_f$ as we wanted to show.

Likewise one shows that $I_t(X^{r\times s}) \in \mathcal{C}_f$,  finishing the proof.
\end{proof}



As a consequence, we conclude  that Schubert determinantal ideals are Knutson ideals for the same choice of $f$, which recovers a result by Knutson \cite{knutson2009frobenius}.

\begin{corollary}\label{c:msk}
Assume \autoref{setup_Knutson_section} and  let $f:=f(X) $ be as in \autoref{def_f(X)}. Let $w \in \mathscr{M}_{k\times \ell}(\mathbb{\kk})$ be a partial  permutation and let $I_{w}$ be its corresponding Schubert determinantal ideal. Then $I_{w} \in \mathcal{C}_f$. In particular,  
$\init_\prec(I_{w})$ is a squarefree monomial ideal for any antidiagonal term order $\prec$ and 
Schubert determinantal  ideals define $F$-pure rings in positive characteristic.
\end{corollary}

\begin{proof}
Notice that $I_{w}$ can be written as a finite sum of ideals of the form $I_t(X_{r\times s})$. Thus, 
the result follows from  \autoref{p:msk}. 
\end{proof}

\begin{remark}
{\it Alternating sign matrix varieties}  and their defining ideals are generalizations of matrix Schubert varieties (see  \cite{klein2021bumpless} for more details).  \autoref{p:msk} applies to these ideals to show that  they belong to  $\mathcal{C}_f$. In particular, their initial ideals with respect to any antidiagonal term order are squarefree and alternating sign matrix ideals define $F$-pure rings in positive characteristic.
\end{remark}

The next result about poset ideals generalizes \cite[Remark 79]{BSV} from maximal minors to minors of any size.

\begin{corollary}\label{c:pif}
Assume \autoref{setup_Knutson_section} and  let $f:=f(X) $ be as in \autoref{def_f(X)}. For any  ideal $\Omega\subseteq \Pi$, the poset ideal $\Omega\kk[X]$  is in $\C_{f}$. In particular, $\init_\prec(\Omega \kk[X])$ is a squarefree monomial ideal for any antidiagonal term order $\prec$ and $\kk[X]/\Omega \kk[X]$ is $F$-pure  in positive characteristic.
\end{corollary}
\begin{proof}
Given $\delta\in\Pi$ we set $\Omega_{\delta}=\{\pi\in\Pi:\pi\not\geq \delta\}\subseteq \Pi$. Notice that $\Omega_{\delta}$ is an ideal of $\Pi$, and we have the equality
$\Omega=\bigcap_{\delta\in\min\{\Pi\setminus \Omega\}}\Omega_{\delta}.$ 
Then, we also have 
$\Omega \kk[X]=\bigcap_{\delta\in\min\{\Pi\setminus \Omega\}}(\Omega_{\delta}\kk[X])$ \cite[Proposition 5.2]{BookDet}. 
If $\delta=[i_1\ldots i_r|j_1\ldots j_r]$, then
\[\Omega_{\delta}\kk[X]=\sum_{t=1}^r\left(I_t\left(X_{[1,j_t-1]}\right)+I_t\left(X^{[1,i_t-1]}\right)\right)+I_{r+1}(X).\]
Therefore, the ideals $\Omega_{\delta}\kk[X]$ belong to $\C_f$ by  \autoref{p:msk}. Thus, their intersection $\Omega \kk[X]$ belongs to $\C_f$ as well. Finally, the conclusion follows as  $\Omega \kk[X]$ is the intersection of ideals in $\C_f$.
\end{proof}

\begin{remark}
There are graded algebras with straightening law over fields of positive characteristic that are not $F$-pure \cite[Remark 5.2]{KoleyVarbaro}, which gives a counterexample  to a conjecture stated in page 245 of \cite{eisenbud80}. However, \autoref{c:pif} confirms the conjecture in many situations.
\end{remark}

\begin{remark}
The fact that for any
antidiagonal term order $\prec$ the ideals $\init_\prec(I_w)$ \cite[Theorem B]{knutson2005grobner}
and  $\init_\prec(\Omega \kk[X])$ \cite[Theorem 2.5]{HerzogTrung} 
are squarefree was already known. 
\end{remark}

We introduce a new notion generalizing that of a poset ideal. Given $1\leq r\leq k, 1\leq s\leq l$, we set $\Pi_{r,s}$ to be the set of all minors of $X_{r\times s}$.
\begin{definition}
A {\it generalized  ideal} of $\Pi$ is a subset $\Omega\subseteq \Pi$ such that 
$$ \text{for all }[i_1\cdots i_r|j_1\cdots j_r]\in\Omega \text{ and }\pi\in\Pi_{i_r,j_r}, \text{ if } \pi\ls [i_1\cdots i_r|j_1\cdots j_r] \text{ then } \pi\in\Omega.$$
\end{definition}

Note that  any ideal of $\Pi$ is a generalized ideal, but there exist  generalized  ideals which are not  ideals themselves, like the sets $\Pi_{r,s}$. 

Since a generalized poset ideal is the sum of  ideals of  the form $\Pi_{r,s}$ (where $r,s$ vary),  the proof of \autoref{c:pif} can be adapted to show the  
following result.

\begin{corollary}\label{c:gpif}
Assume \autoref{setup_Knutson_section} and  let $f:=f(X) $ be as in \autoref{def_f(X)}. For any generalized ideal $\Omega\subseteq \Pi$, the generalized poset ideal $\Omega \kk[X]$  is in $\C_{f}$. 
In particular, $\init_\prec(\Omega \kk[X])$ is a squarefree monomial ideal for any antidiagonal term order $\prec$, and $\kk[X]/\Omega \kk[X]$ is $F$-pure  in positive characteristic.
\end{corollary}

\section*{Acknowledgments}
The first and fifth authors were partially supported by the PRIN 2020 project \#2020355B8Y ``Squarefree Gröbner degenerations, special varieties and related topics'' and by the MIUR Excellence Department Project, CUP D33C23001110001. The first, fourth and fifth authors were partially supported by INdAM-GNSAGA. The second author was partially funded by NSF Grant DMS \#2001645/2303605. 
The third author   was  supported by CONACyT Grant \#284598, 
and by SECIHTI  Grants CBF 2023-2024-224 \& CF-2023-G-33.
We thank the anonymous referees for helpful suggestions.
\bibliographystyle{alpha}
\bibliography{References}

\end{document}